\documentclass[11pt]{amsart}

\usepackage{amsmath,amssymb,amsthm}
\usepackage[colorlinks=true,linkcolor=blue,citecolor=blue,urlcolor=blue]{hyperref}
\theoremstyle{plain}
\newtheorem{theorem}{Theorem}[section]
\newtheorem{lemma}[theorem]{Lemma}
\newtheorem{corollary}[theorem]{Corollary}

\theoremstyle{definition}
\newtheorem{definition}[theorem]{Definition}
\newtheorem{example}[theorem]{Example}

\theoremstyle{remark}

\numberwithin{equation}{section}

\newcommand{\N}{\mathbb N}
\newcommand{\Nzero}{\mathbb N_0}
\newcommand{\R}{\mathbb R}
\newcommand{\Pp}{\mathbb P}
\newcommand{\E}{\mathbb E}
\newcommand{\dd}{\stackrel{d}{=}}
\newcommand{\norm}[1]{\left\lVert #1\right\rVert}

\begin{document}

\title[Almost periodicity as a path property for \texorpdfstring{$p$-adic}{} sssi processes]
{Almost periodicity as a path property for \texorpdfstring{\protect\(p\protect\)-adic}{} self-similar processes with stationary increments}

\author{Yi Shen}

\author{Zhenyuan Zhang}

\date{\today}


\begin{abstract}
Shen and Zhang (2021) showed that almost periodicity naturally arises in the
spectral representation of discrete-time \(p\)-adic self-similar processes with
stationary increments. In this paper, we study several notions of almost
periodicity as sample path properties of Banach space-valued \(p\)-adic sssi
processes. We prove that Bohr almost periodicity is equivalent, as a path
event, to continuity with respect to the \(p\)-adic topology. We also show that
the corresponding equivalence fails for Weyl and Besicovitch almost
periodicity. Finally, we extend the Bohr almost-periodic result to
finite-dimensional random fields.
\end{abstract}

\maketitle

\section{Introduction}

Self-similar processes with stationary increments, or sssi processes, are
random processes whose laws are controlled by two basic symmetries: dilating
time and translating the time origin for increments.  In the usual continuous-time setting,
self-similarity says that a dilation of time changes the law only by a
deterministic power factor, while stationary increments say that the law of an
increment depends on its length but not on its base point.  These symmetries
underlie standard examples such as fractional Brownian motion, stable L\'{e}vy processes,
and many long-range dependent models~\cite{EmbrechtsMaejima2002,SamorodnitskyTaqqu1994,PipirasTaqqu2017}.  In the discrete-time setting studied by \cite{GefferthEtAl2003,ShenZhang2021}, the scaling factor need not be governed
by the ordinary absolute value of an integer.  It may instead be governed by
the \(p\)-adic norm: dilating time by a nonzero integer \(a\) changes the law
by the factor \(|a|_p^H\) for some $H>0$.  This gives a natural class of
\textit{\(p\)-adic sssi processes}, whose sample paths are sequences indexed by
\(\mathbb N_0\) but whose scaling symmetry follows the \(p\)-adic topology.\footnote{In \cite{ShenZhang2021}, $p$-adic sssi processes are called discrete-time sssi processes of type II.}

A distinguishing feature of \(p\)-adic sssi processes in $L^2$ is that they are mean-square Bohr almost periodic.\footnote{Recall that a stochastic process $(X_n)_{n\in \Nzero}$ in $L^2$ is mean-square Bohr almost periodic, if for
every \(\varepsilon>0\), there exists \(L\in\mathbb N\) such that every interval
of \(L\) consecutive positive integers contains some translation number \(\tau\) satisfying 
   $\sup_{n\ge0}\mathbb E[|X_{n+\tau}-X_n|]^2<\varepsilon$.}  Leveraging the mean-square almost periodicity, \cite{ShenZhang2021} proved that \(p\)-adic sssi processes in $L^2$ admit a Fourier-type
spectral representation. This is similar to the standard spectral representation of stationary processes, but the difference is that, here, only \(p\)-adic rational
frequencies \(\ell/p^m\) appear and the Fourier coefficients satisfy distributional relations coming from the probabilistic symmetries.

In this paper, we study almost periodicity as a sample path property of
\(p\)-adic sssi processes. The  general notions of almost periodicity refer to relative denseness of translation numbers $\tau$ under which the function changes little, either uniformly or in average senses. A central question is, for different versions $\Xi$ of almost periodicity, 
\begin{center}
    \textit{is $\Xi$-almost periodicity equivalent to continuity w.r.t.~the $p$-adic topology\\ as sample path properties, for $p$-adic sssi processes?}
\end{center}
Here, $\Xi$-almost periodicity may refer to almost periodicity in the Bohr, Weyl, or Besicovitch sense \cite{Bohr1925I,Weyl1927,Besicovitch1926}, as well as extensions such as limit periodicity and asymptotic almost periodicity; see Definition~\ref{def} and Figure \ref{fig:deterministic-hierarchy} below. 

Our main results indicate that the answer is mixed. We show that \(p\)-adic continuity, Bohr almost periodicity, limit periodicity, and asymptotic almost periodicity are equivalent as sample path properties for \(p\)-adic sssi processes (Theorems \ref{thm:main1} and \ref{thm:asymp p}, and Corollary \ref{coro:limit p}). However, surprisingly, the same does not hold for Weyl almost periodicity or Besicovitch almost periodicity (Theorem \ref{thm:WB p}).

The easy direction of the equivalence is deterministic: if \(f\) is uniformly \(p\)-adically
continuous, then every integer with a sufficiently small $p$-adic norm is a valid translation number, so
\(f\) satisfies those almost periodic conditions.  The other direction is harder and
uses the probabilistic symmetries.  
The intuition is that the \(p\)-adic scaling symmetry strongly biases good
translation numbers toward high \(p\)-adic divisibility. This is because the \(p\)-adic sssi property
suggests that increments along sublattices \(r+p^K\mathbb N_0\) have the same law as the original process after scaling by $p^{-KH}$, and hence are stochastically small. On the other hand, Weyl and Besicovitch almost periodicity control translates only
in averaged seminorms, so rare large increments can be invisible to the
almost-periodic condition while still destroying \(p\)-adic continuity, which
requires uniform control over all indices. \\

\textbf{Related literature.}
We now place these results in context. First, we use classical facts from the theory of almost-periodic functions and
sequences; standard references include
\cite{AmerioProuse1971,Corduneanu1989,LevitanZhikov1982},
and broader hierarchies of almost-periodic spaces are surveyed in
\cite{Grande2006}. 

Second, the probabilistic background traces back to Lamperti's
formulation of self-similarity~\cite{Lamperti1962}.  On the real line,
sample path properties of sssi processes were systematically studied by~\cite{Vervaat1985}, with stable sssi examples treated in
\cite{KonoMaejima1991,SamorodnitskyTaqqu1994}.  For \(p\)-adic sssi
processes,~\cite{ShenZhang2021} provided classification
results, tree constructions, and an \(L^2\) spectral representation.
 Our result complements the sample path viewpoint of~\cite{Vervaat1985}, but with a different emphasis: instead of assuming a
regular path class in advance, we show that the \(p\)-adic sssi symmetries
upgrade the qualitative assumption of Bohr almost periodicity to uniform
\(p\)-adic continuity.

Third, there is also a separate literature on almost-periodic stochastic processes.
One line, closer in spirit to path questions, studies sufficient conditions for
sample paths of stationary processes to be almost periodic in
Besicovitch, Stepanov, or uniform senses
\cite{Slutsky1938,Udagawa1952,Kawata1982}.  Another line studies
almost periodicity in distribution, in square mean, or in \(p\)-th mean, often
as a regularity condition for solutions of stochastic difference, differential,
and evolution equations
\cite{DaPratoTudor1995,BezandryDiagana2011}.  Our work
is different from both lines: almost periodicity is imposed directly on
individual sample paths, and the goal is to identify the deterministic path
class forced by the $p$-adic sssi symmetries.\\

\textbf{Paper outline.} 
The rest of this paper is organized as follows.  Section~\ref{sec:conventions} discusses in detail the
notions of almost periodicity and continuity and records the deterministic
hierarchy used later.  Section~\ref{sec:one-dimensional} proves the
equivalence theorem for Bohr almost periodicity and limit periodicity.  Section~\ref{sec:variants} treats the
asymptotically almost-periodic variant.  Section~\ref{sec:counterexamples}
constructs heavy-tailed examples of $p$-adic sssi processes, showing that the equivalence theorem does not hold for Weyl and
Besicovitch almost periodicity.  Section~\ref{sec:multi}
gives an extension of the equivalence theorem to random fields.

\section{Basic Setting and Definitions}\label{sec:conventions}

Throughout, \(p\) is a fixed prime number and \(H>0\). Let $B$ be a separable Banach space (and hence a Polish space).  Following \cite{ShenZhang2021}, a \(B\)-valued
\emph{\(p\)-adic \(H\)-sssi process} is a process
\(X=(X_n)_{n\in\Nzero}\) such that
\begin{equation}\label{eq:sssi-1d}
   (X_{a n})_{n\in\Nzero}\dd |a|_p^H(X_n)_{n\in\Nzero},
   \qquad
   (X_{n+\ell}-X_\ell)_{n\in\Nzero}\dd (X_n-X_0)_{n\in\Nzero}
\end{equation}
for all \(a\in\N\) and \(\ell\in\Nzero\), where $\N$ and $\Nzero$ are the sets of positive and non-negative integers, respectively, and $|\cdot|_p$ denotes the $p$-adic norm.

By \cite[Proposition 3.11(i)]{ShenZhang2021}, $X_0=0$ a.s.
It is then easy to check by definition that any \(p\)-adic \(H\)-sssi process is stochastically continuous with respect to the $p$-adic topology on $\Nzero$. Therefore, continuity with respect to the $p$-adic topology becomes a natural sample path property for such processes. For a deterministic function $f:\N_0\to B$, we say $f$ is \emph{(uniformly) $p$-adically continuous} if $f$ is uniformly continuous with respect to the $p$-adic topology, i.e.,
\begin{equation}\label{eq:padic-cont-1d}
   \forall\varepsilon>0,\ \exists K,\quad
   \sup_{n,u\geq0}\norm{f(n+p^K u)-f(n)}<\varepsilon.
\end{equation}


Next, we collect the almost-periodicity notions used throughout the paper. Recall that a set
\(T\subset\N\) is \emph{relatively dense} if there is \(L\in\N\) such that every
interval of \(L\) consecutive positive integers contains an element of \(T\).
\begin{definition}\label{def}
Let \(B\) be a Banach space and let \(f:\Nzero\to B\).  

\begin{enumerate}
\item A positive integer \(\tau\) is a
\emph{Bohr \(\varepsilon\)-translation number}, or simply an \emph{\(\varepsilon\)-translation number} of \(f\), if
\[
   \sup_{n\ge0}\norm{f(n+\tau)-f(n)}<\varepsilon .
\]
The sequence \(f\) is \emph{(Bohr) almost periodic}, if for every
\(\varepsilon>0\), the set of Bohr \(\varepsilon\)-translation numbers is
relatively dense.


\item The sequence \(f:\Nzero\to B\) is \emph{limit periodic} if
it is a uniform limit of periodic sequences, equivalently if for
every \(\varepsilon>0\) there is a periodic sequence \(g:\Nzero\to B\) such
that
\[
   \sup_{n\ge0}\norm{f(n)-g(n)}<\varepsilon .
\]

\item The sequence \(f:\Nzero\to B\) is
\emph{asymptotically almost periodic} if
\[
   f(n)=g(n)+c(n),\qquad n\ge0,
\]
where \(g:\Nzero\to B\) is almost periodic and \(c(n)\to0\) in \(B\).

\item For \(q\ge1\), define the \emph{Weyl \(q\)-seminorm} of a sequence
\(u:\Nzero\to B\) by
\[
   \norm{u}_{W^q}
   :=
   \limsup_{L\to\infty}\sup_{N\ge0}
   \left(\frac1L\sum_{n=N}^{N+L-1}\norm{u(n)}^q\right)^{1/q}.
\]
A positive integer \(\tau\) is a
\emph{Weyl \(q\)-\(\varepsilon\)-translation number} of \(f\) if
\[
   \norm{\,f(\cdot+\tau)-f(\cdot)\,}_{W^q}<\varepsilon .
\]
The sequence \(f\) is \emph{Weyl \(q\)-almost periodic} if, for every
\(\varepsilon>0\), the set of Weyl \(q\)-\(\varepsilon\)-translation numbers is
relatively dense.

\item For \(q\ge1\), define the \emph{Besicovitch \(q\)-seminorm} by
\[
   \norm{u}_{B^q}
   :=
   \limsup_{L\to\infty}
   \left(\frac1L\sum_{n=0}^{L-1}\norm{u(n)}^q\right)^{1/q}.
\]
A positive integer \(\tau\) is a
\emph{Besicovitch \(q\)-\(\varepsilon\)-translation number} of \(f\) if
\[
   \norm{\,f(\cdot+\tau)-f(\cdot)\,}_{B^q}<\varepsilon .
\]
The sequence \(f\) is \emph{Besicovitch \(q\)-almost periodic} if, for every
\(\varepsilon>0\), the set of Besicovitch \(q\)-\(\varepsilon\)-translation
numbers is relatively dense.
\end{enumerate}
\end{definition}

The deterministic notions of almost periodicity above are related as shown in
Figure~\ref{fig:deterministic-hierarchy}. First, if \(f\) is \(p\)-adically continuous, let \(r_K(n)\) be the least residue of \(n\) modulo \(p^K\) and put \(f_K(n)=f(r_K(n))\).  Then
 \(f_K\) is \(p^K\)-periodic and \(f_K\to f\) uniformly.  Hence continuity
 implies limit periodicity. The remaining arrows are standard parts of the deterministic hierarchy; see the
survey~\cite{Grande2006}. 

\begin{figure}[t]
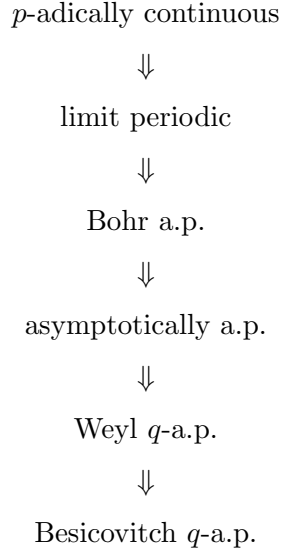

\[
\begin{array}{c}
\text{\(p\)-adically continuous}
\\[0.6em]
\Downarrow
\\[0.6em]
\text{limit periodic}
\\[0.6em]
\Downarrow
\\[0.6em]
\text{Bohr a.p.}
\\[0.6em]
\Downarrow
\\[0.6em]
\text{asymptotically a.p.}
\\[0.6em]
\Downarrow
\\[0.6em]
\text{Weyl \(q\)-a.p.}
\\[0.6em]
\Downarrow
\\[0.6em]
\text{Besicovitch \(q\)-a.p.}
\end{array}
\]
\caption{Deterministic implications among the almost-periodicity (a.p.) and
continuity notions used in the paper, for fixed \(q\geq1\).  The reverse implications are false in general.}
\label{fig:deterministic-hierarchy}
\end{figure}

In this paper, we study $p$-adic continuity and various notions of almost periodicity as sample path properties. It is clear from the definitions that the path events are all measurable.

We end this section with a simple example that shows that deterministically, limit periodicity is indeed weaker than $p$-adic continuity, hence all the other almost periodicities mentioned above are also weaker than $p$-adic continuity.

\begin{example}
Let \(p'\ge 2\) be coprime to \(p\), and define \(f(n)=\mathbf 1_{\{p'\mid n\}}\) for \( n\in\mathbb N_0\). Then \(f\) is \(p'\)-periodic, and hence limit periodic. However, for every \(K\ge 0\), we have
\[
\sup_{n,u\ge 0}|f(n+p^K u)-f(n)|
\ge |f(p^K)-f(0)|=1.
\]
Thus,  \(f\) is not \(p\)-adically continuous.
\end{example}

\section{One-dimensional equivalence of \texorpdfstring{$p$-adic}{} continuity and almost periodicity}\label{sec:one-dimensional}

The goal of this section is to prove Theorem \ref{thm:main1} below, which asserts the equivalence of $p$-adic continuity and almost periodicity under $\Pp$ for $p$-adic sssi processes. We start from two elementary lemmas on the behavior of almost
periodic sequences under translations and scaling.

\begin{lemma}\label{lem:finite1}
Let \(B\) be a Banach space and let \(f:\Nzero\to B\) be almost
periodic.  For every \(\eta>0\) there is \(L=L(f,\eta)\) such that, for
every \(h\geq0\), some \(r\in\{0,\dots,L\}\) satisfies
\begin{equation*}
   \sup_{m\geq0}
   \norm{(f(m+h)-f(h))-(f(m+r)-f(r))}<\eta .
\end{equation*}
Consequently, for every \(a\geq0\),
\[
   \sup_{h\geq0}\norm{f(h+a)-f(h)}
   \leq \eta+\max_{0\leq r\leq L}\norm{f(r+a)-f(r)} .
\]
\end{lemma}

\begin{proof}
By the definition of almost periodicity, choose \(L\) so that every interval of \(L\) consecutive positive
integers contains an \(\eta/2\)-translation number of \(f\).  If
\(h\leq L\), take \(r=h\).  If \(h>L\), choose such a translation number
\(\tau\in[h-L,h]\) and put \(r=h-\tau\).  Then \(0\leq r\leq L\), and
for all \(m\geq0\), $\norm{f(m+h)-f(m+r)}<\eta/2$ and $\norm{f(h)-f(r)}<\eta/2$. 
The first assertion then follows from the triangle inequality, and the second is obtained by letting \(m=a\).
\end{proof}

\begin{lemma}\label{lem:arith1}
Let \(B\) be a Banach space and let \(f:\Nzero\to B\) be almost
periodic.  Then for every \(r\geq0\) and \(q\geq1\),
\[
   u\mapsto f(r+qu)-f(r),\qquad u\geq0,
\]
is almost periodic.
\end{lemma}

\begin{proof}
By the Banach space-valued Bohr approximation theorem, \(f\) is a uniform
limit on \(\Nzero\) of trigonometric polynomials; see~\cite{Bochner1933}. Thus, \(f\) can be uniformly approximated by finite sums of the form
\[
   P(n)=b_0+\sum_{j=1}^J
   \bigl(b_j\cos(2\pi\lambda_j n)+c_j\sin(2\pi\lambda_j n)\bigr),
   \qquad b_j,c_j\in B .
\]
For such a polynomial, elementary angle-addition identities show that
\(P(r+qu)-P(r)\) is again a \(B\)-valued trigonometric polynomial in
\(u\).  The claim then follows since uniform limits of almost periodic sequences are almost periodic.
\end{proof}

The following conditioning lemma is the probabilistic device that lets the
sssi identities survive after conditioning on a path event such as almost periodicity. 

\begin{lemma}\label{lem:cond}
Let \(S\) be a measurable path space, let \(A\subset S\) be measurable,
and let \(T,S_c:S\to S\) be measurable maps.  Suppose \(TX\dd S_cX\),
\(\Pp(X\in A)>0\), \(A\subseteq T^{-1}A\), and
\(\{x:S_cx\in A\}=A\).  Then, under \(\mathbb{Q}=\Pp(\,\cdot\,|X\in A)\),
\[
   TX\dd S_cX .
\]
\end{lemma}

\begin{proof}
The hypotheses give
\[
   \Pp(TX\in A)=\Pp(S_cX\in A)=\Pp(X\in A).
\]
Since \(A\subseteq T^{-1}A\), \(A=T^{-1}A\) modulo the law of \(X\).
Thus, for every bounded measurable function \(F\),
\[
\begin{aligned}
   \E_{\mathbb{Q}}[F(TX)]
   =\frac{\E[F(TX)\mathbf1_{\{X\in A\}}]}{\Pp(X\in A)}
     &=\frac{\E[F(TX)\mathbf1_{\{TX\in A\}}]}{\Pp(X\in A)}\\
   &=\frac{\E[F(S_cX)\mathbf1_{\{S_cX\in A\}}]}{\Pp(X\in A)}
     =\E_{\mathbb{Q}}[F(S_cX)].
\end{aligned}
\]
This proves the claimed conditional distributional identity.
\end{proof}


We now prove the main result of this section. Let $E_1\Delta E_2$ denote the symmetric difference of events $E_1$ and $E_2$. 
\begin{theorem}\label{thm:main1}
Let \(B\) be a separable Banach space and let
\((X_n)_{n\geq0}\) be a \(B\)-valued \(p\)-adic \(H\)-sssi process.
Then
\[
  \Pp(\{X\hbox{ is almost periodic}\}\Delta \{X\hbox{ is }p\hbox{-adically continuous}\})
  =0.
\]
\end{theorem}

\begin{proof}
Continuity implies almost periodicity, which follows from the deterministic relation in Figure \ref{fig:deterministic-hierarchy}. We provide details here for completeness. If
\[
   \sup_{n,u\geq0}\norm{f(n+p^K u)-f(n)}<\varepsilon ,
\]
then every positive multiple \(\tau=p^K u\) is an
\(\varepsilon\)-translation number:
\[
   \sup_{n\geq0}\norm{f(n+\tau)-f(n)}<\varepsilon.
\]
The set \(p^K\N\) has gaps equal to \(p^K\), and hence is relatively dense.

For the converse, let \(A\) be the event that \(X(\cdot)\) is almost
periodic.  If \(\Pp(A)=0\), there is nothing to prove.  Otherwise, work
under the conditional probability \(\mathbb{Q}=\Pp(\cdot\mid A)\).

Let $C$ be the event that $X\hbox{ is }p\hbox{-adically continuous}$. It remains to show that $\Pp(A\cap C^c)=0$. By \eqref{eq:padic-cont-1d},
\begin{equation}\label{eq:C-from-CK-1d}
   C=\bigcap_{j=1}^{\infty}\bigcup_{K=0}^{\infty}C_K\Big(\frac{1}{j}\Big),
\end{equation}
where we define
\begin{equation*}
   C_K(\varepsilon)=
   \left\{\sup_{n,u\geq0}\norm{X_{n+p^K u}-X_n}<\varepsilon\right\}.
\end{equation*}

Suppose for the moment that for any $\varepsilon>0$, $\mathbb{Q}(C_K(\varepsilon))\to 1$ as $K\to\infty$. 
Because \(C_K(\varepsilon)\) is increasing in \(K\),
\[
   \mathbb{Q}\Big(\bigcup_K C_K(\varepsilon)\Big)=1 .
\]
Therefore, \eqref{eq:C-from-CK-1d} gives $\mathbb{Q}(C)=1$, and hence $\Pp(A\cap C^c)=\Pp(A)\,\mathbb{Q}(C^c)=0$. 
This completes the proof. 

In the rest of the proof, we fix $\varepsilon>0$ and show that $\mathbb{Q}(C_K(\varepsilon))\to 1$ as $K\to\infty$. 
Put \(\eta=\varepsilon/2\), and define the event
\begin{equation*}
\begin{aligned}
   D_L(\eta)=
   \bigl\{&\forall h\geq0,\ \exists r\in\{0,\dots,L\},\\
   &\sup_{m\geq0}\norm{(X_{m+h}-X_h)-(X_{m+r}-X_r)}<\eta
   \bigr\}.
\end{aligned}
\end{equation*}
 By Lemma~\ref{lem:finite1},
\(\mathbb{Q}(\bigcup_LD_L(\eta))=1\), and the events
\(D_L(\eta)\) are increasing.  Given \(\gamma>0\), choose \(L_0\) with
\(\mathbb{Q}(D_{L_0}(\eta))>1-\gamma\).  On the event $C_K(\varepsilon)^c\cap   D_{L_0}(\eta)$, there exist $n,u\geq 0$ such that $\norm{X_{n+p^K u}-X_n}\geq\varepsilon$. Moreover, there exists \(r\in\{0,\dots,L_0\}\) such that
\[
   \sup_{m\geq0}
   \norm{(X_{m+n}-X_n)-(X_{m+r}-X_r)}<\frac{\varepsilon}{2} .
\]
Taking \(m=p^K u\) gives
\[
\begin{aligned}
   \norm{X_{n+p^K u}-X_n}
   &\leq
   \norm{(X_{n+p^K u}-X_n)-(X_{r+p^K u}-X_r)}
      +\norm{X_{r+p^K u}-X_r}\\
   &<\frac{\varepsilon}{2}+\norm{X_{r+p^K u}-X_r}.
\end{aligned}
\]
Therefore,
\begin{equation}\label{eq:C-inclusion-1d}
   C_K(\varepsilon)^c
   \subseteq
   D_{L_0}(\eta)^c
   \cup
   \bigcup_{r=0}^{L_0}\left\{\sup_{u\geq0}\norm{X_{r+p^K u}-X_r}\geq\frac{\varepsilon}{2}\right\}.
\end{equation}

\sloppy Next, we control the probabilities of the events $\left\{\sup_{u\geq0}\norm{X_{r+p^K u}-X_r}\geq\frac{\varepsilon}{2}\right\}$. Since almost periodic sequences are bounded, we can define
\[
   M=\sup_{n\geq0}\norm{X_n}<\infty,\qquad \mathbb{Q}\hbox{-a.s.}
\]
For \(r,K\geq0\), define the transformation
\[
   \Phi_{r,K}X=(X_{r+p^K u}-X_r)_{u\geq0}.
\]
By the two identities in \eqref{eq:sssi-1d}, we have
\begin{equation}\label{eq:phi-law-1d}
   \Phi_{r,K}X\dd p^{-KH}X.
\end{equation}
By Lemma~\ref{lem:arith1}, \(A\subseteq\{\Phi_{r,K}X\in A\}\).  Since
nonzero scalar multiplication preserves almost periodicity, Lemma
\ref{lem:cond} and \eqref{eq:phi-law-1d} give, under \(\mathbb{Q}\),
\begin{equation*}
   \sup_{u\geq0}\norm{X_{r+p^K u}-X_r}\dd p^{-KH}M .
\end{equation*}
Thus, for every fixed \(r\) and \(\delta>0\),
\begin{equation}\label{eq:Y-tail-1d}
   \mathbb{Q}\left(\sup_{u\geq0}\norm{X_{r+p^K u}-X_r}\geq\delta\right)
   =\mathbb{Q}(p^{-KH}M\geq\delta)
   =\mathbb{Q}(M\geq\delta p^{KH}).
\end{equation}

Combining \eqref{eq:C-inclusion-1d} with \eqref{eq:Y-tail-1d} and the union bound gives
\begin{equation*}
\begin{aligned}
   \mathbb{Q}(C_K(\varepsilon)^c)
   &\leq \mathbb{Q}(D_{L_0}(\eta)^c)
      +\sum_{r=0}^{L_0}\mathbb{Q}
      \left(\sup_{u\geq0}\norm{X_{r+p^K u}-X_r}\geq\frac{\varepsilon}{2}\right)\\
   &< \gamma+\sum_{r=0}^{L_0}\mathbb{Q}\left(M\geq \frac{\varepsilon p^{KH}}{2}\right)\\
   &\leq \gamma+(L_0+1)\mathbb{Q}\left(M\geq \frac{\varepsilon p^{KH}}{2}\right).
\end{aligned}
\end{equation*}
Thus, 
\[
   \limsup_{K\to\infty}\mathbb{Q}(C_K(\varepsilon)^c)\leq\gamma .
\]
Letting \(\gamma\downarrow0\) gives \(\mathbb{Q}(C_K(\varepsilon))\to1\). This completes the proof.
\end{proof}

As a consequence, the limit-periodic path event has the same probability as the \(p\)-adic
continuity event, which follows directly from Theorem \ref{thm:main1} and the deterministic relations described in Figure \ref{fig:deterministic-hierarchy}.
\begin{corollary}\label{coro:limit p}
Let \(X=(X_n)_{n\geq0}\) be a Banach space-valued \(p\)-adic \(H\)-sssi process.  Then
\[
   \Pp(\{X\hbox{ is limit periodic}\}\Delta\{X\hbox{ is }p\hbox{-adically continuous}\})
   =0.
\]
\end{corollary}



\section{Asymptotically Almost-Periodic Paths}
\label{sec:variants}

In this section, we prove that $p$-adic continuity and asymptotic almost periodicity are equivalent under $\Pp$ for $p$-adic sssi processes. Recall that a Banach space-valued sequence \(f:\Nzero\to B\) is asymptotically almost periodic
if \(f(n)=g(n)+c(n)\), \(n\geq0\), where \(g\) is almost periodic and
\(c(n)\to0\) in \(B\). We follow a path similar to the proof of Theorem \ref{thm:main1}, starting with analogues of Lemmas \ref{lem:finite1} and \ref{lem:arith1}.

\begin{lemma}\label{lem:aap_finite}
Let \(B\) be a Banach space and let \(f:\Nzero\to B\) be
asymptotically almost periodic.  For every \(\eta>0\) there is
\(L=L(f,\eta)\) such that, for every \(h\geq0\), some
\(r\in\{0,\dots,L\}\) satisfies
\begin{equation}
    \sup_{m\geq0}
   \norm{(f(m+h)-f(h))-(f(m+r)-f(r))}<\eta .\label{eq:ts}
\end{equation}
\end{lemma}

\begin{proof}
Write \(f=g+c\), where \(g\) is almost periodic and \(c(n)\to0\).  Choose
\(R\geq1\) so that
\[
   n\geq R\quad\Longrightarrow\quad \norm{c(n)}<\frac{\eta}{8} .
\]
By Lemma~\ref{lem:finite1} applied to \(g\) with error \(\eta/4\), there exists
\(L_g\) such that for every \(h\geq0\) there is
\(s\in\{0,\dots,L_g\}\) with
\begin{equation}
    \sup_{m\geq0}
   \norm{(g(m+h)-g(h))-(g(m+s)-g(s))}<\frac{\eta}{4} .\label{eq:s1}
\end{equation}
For each \(s\in\{0,\dots,L_g\}\), choose an \(\eta/8\)-translation number
\(\tau_s\) of \(g\) so large that \(s+\tau_s\geq R\). 
It follows that
\begin{equation}
\begin{aligned}
&\sup_{m\geq0}
\norm{(g(m+s+\tau_s)-g(s+\tau_s))-(g(m+s)-g(s))}\\
&\qquad\leq
\sup_{m\geq0}\norm{g(m+s+\tau_s)-g(m+s)}
   +\norm{g(s+\tau_s)-g(s)}
<\frac{\eta}{4} .
\end{aligned}
\label{s2}
\end{equation}
Let $L=\max\{R-1,L_g+\max\{\tau_0,\dots,\tau_{L_g}\}\}$. 
If \(h<R\), take \(r=h\) and \eqref{eq:ts} is trivial.  If \(h\geq R\), choose \(s\) as above and take
\(r=s+\tau_s\) so $r\leq L$.  For this \(r\), \eqref{eq:s1} and \eqref{s2} together yield
\begin{equation}
    \sup_{m\geq0}
   \norm{(g(m+h)-g(h))-(g(m+r)-g(r))}<\frac{\eta}{2}.\label{eq:t1}
\end{equation}
Also \(h,m+h,r,m+r\geq R\), so
\begin{equation}
    \label{eq:t2}
\begin{aligned}
&\sup_{m\geq0}
\norm{(c(m+h)-c(h))-(c(m+r)-c(r))}\\
&\qquad\leq
\sup_{m\geq0}\bigl(\norm{c(m+h)}+\norm{c(h)}
                  +\norm{c(m+r)}+\norm{c(r)}\bigr)
<\frac{\eta}{2} .
\end{aligned}
\end{equation}
Combining \eqref{eq:t1} and \eqref{eq:t2} proves the claim since $f=g+c$.
\end{proof}

\begin{lemma}\label{lem:aap_arith}
If \(f:\Nzero\to B\) is asymptotically almost periodic, then for every
\(r\geq0\) and \(q\geq1\), the sequence \(u\mapsto f(r+qu)-f(r)\) is
asymptotically almost periodic.
\end{lemma}

\begin{proof}
Write \(f=g+c\), with \(g\) almost periodic and \(c(n)\to0\).  By
Lemma~\ref{lem:arith1}, \(u\mapsto g(r+qu)-g(r)\) is almost periodic.
Moreover, \(u\mapsto -c(r)\) is constant, and
\(c(r+qu)\to0\).  Thus, 
\[
   f(r+qu)-f(r)
   =
   \bigl(g(r+qu)-g(r)-c(r)\bigr)+c(r+qu)
\]
is the sum of an almost periodic sequence and a sequence converging to
\(0\).
\end{proof}

\begin{theorem}\label{thm:asymp p}
Let \(B\) be a separable Banach space and let
\((X_n)_{n\geq0}\) be a \(B\)-valued \(p\)-adic \(H\)-sssi process.  Then
\[
  \Pp(\{X\hbox{ is asymptotically almost periodic}\}\Delta\{X\hbox{ is }p\hbox{-adically continuous}\})
  =0.
\]
\end{theorem}

\begin{proof}
The proof follows the same lines as in the proof of Theorem~\ref{thm:main1}, with
Lemma~\ref{lem:aap_finite} replacing Lemma~\ref{lem:finite1} and Lemma~\ref{lem:aap_arith} replacing Lemma~\ref{lem:arith1}, while noting that almost periodicity implies asymptotic almost periodicity. In applying the
conditioning step, we use that asymptotically almost-periodic paths are bounded,
that the asymptotically almost-periodic path event is invariant under nonzero
scalar multiplication, and that Lemma~\ref{lem:aap_arith} gives the required
stability under sublattice recentering. All of these assertions are easy to check and are omitted.
\end{proof}

\section{Counterexamples to Weyl and Besicovitch versions}
\label{sec:counterexamples}

In this section, we show that the equivalence of $p$-adic continuity and almost periodicity cannot be
extended to Weyl or Besicovitch almost periodicity.  The obstruction is that
these notions average over \(n\), whereas
\(p\)-adic continuity is a uniform supremum condition.  A heavy-tailed
\(p\)-adic tree makes the average tail small but leaves arbitrarily large
extreme increments at every small \(p\)-adic scale.

The next lemma provides a class of $p$-adic sssi processes, following \cite[Example 4.1]{ShenZhang2021}. 
\begin{lemma}\label{lem:tree_sssi}
Let \((\xi_{k,r})_{k\geq0,\ 0\leq r<p^{k+1}}\) be independent copies of a
real random variable \(\xi\), and extend periodically by
\[
   \xi_{k,n}=\xi_{k,r}\quad\hbox{if }n\equiv r\pmod {p^{k+1}} .
\]
Assume \(\E[|\xi|]<\infty\) and let \(H>0\).  Then
\begin{equation}\label{eq:tree-series}
   X_n=\sum_{k=0}^{\infty}p^{-kH}(\xi_{k,n}-\xi_{k,0}),
   \qquad n\geq0,
\end{equation}
converges absolutely a.s.~for every fixed \(n\), and
\((X_n)_{n\geq0}\) is a real-valued \(p\)-adic \(H\)-sssi process.
\end{lemma}


\begin{theorem}\label{thm:WB p}
Fix any \(q\geq1\) and $H>0$.  There is a real-valued \(p\)-adic \(H\)-sssi process
which is a.s.~Weyl \(q\)-almost periodic, hence Besicovitch
\(q\)-almost periodic, but is a.s.~not \(p\)-adically continuous.
\end{theorem}

\begin{proof}
Choose $\alpha\in((H+1/q)^{-1},H^{-1})$ and let \(\xi\) be a
symmetric random variable such that
\[
   \Pp(|\xi|>t)=t^{-\alpha},\qquad t\geq1 .
\]
Construct \(X\) from Lemma~\ref{lem:tree_sssi}, using
\eqref{eq:tree-series}.  Since \(\alpha>q\geq1\), \(\E[|\xi|]<\infty\),
so the construction is legitimate.

First we prove Weyl \(q\)-almost periodicity.  For \(K\geq1\) and
\(\tau\in p^K\N\), all levels \(k<K\) cancel:
\begin{equation}\label{eq:tree-cancellation}
   X_{n+\tau}-X_n
   =
   \sum_{k=K}^{\infty}p^{-kH}(\xi_{k,n+\tau}-\xi_{k,n}).
\end{equation}
For each \(k\), define the period average
\[
   A_{k,q}(\tau)
   =
   \left(
   p^{-(k+1)}
   \sum_{r=0}^{p^{k+1}-1}|\xi_{k,r+\tau}-\xi_{k,r}|^q
   \right)^{1/q}.
\]
The sequence \(n\mapsto \xi_{k,n+\tau}-\xi_{k,n}\) is periodic with
period \(p^{k+1}\), so a.s.
\begin{equation}
    \lim_{L\to\infty}\sup_{N\geq0}
   \left(
   \frac1L\sum_{n=N}^{N+L-1}
   |\xi_{k,n+\tau}-\xi_{k,n}|^q
   \right)^{1/q}
   =A_{k,q}(\tau).\label{eq:t3}
\end{equation}
Combining \eqref{eq:tree-cancellation} and \eqref{eq:t3} with Minkowski's inequality gives
\begin{equation}\label{eq:weyl-upper-bound}
\begin{aligned}
&\limsup_{L\to\infty}\sup_{N\geq0}
\left(
\frac1L\sum_{n=N}^{N+L-1}|X_{n+\tau}-X_n|^q
\right)^{1/q}\leq
\sum_{k=K}^{\infty}p^{-kH}A_{k,q}(\tau).
\end{aligned}
\end{equation}
By Minkowski's inequality on the finite cyclic group
\(\mathbb Z/p^{k+1}\mathbb Z\), we have
\begin{equation}
    \label{eq:c1}
\begin{aligned}
   A_{k,q}(\tau)
   &\leq
   \left(p^{-(k+1)}
   \sum_{r=0}^{p^{k+1}-1}|\xi_{k,r+\tau}|^q\right)^{1/q}\\
   &\qquad+
   \left(p^{-(k+1)}
   \sum_{r=0}^{p^{k+1}-1}|\xi_{k,r}|^q\right)^{1/q}
   =2B_{k,q},
\end{aligned}
\end{equation}
where
\[
   B_{k,q}=
   \left(
   p^{-(k+1)}
   \sum_{r=0}^{p^{k+1}-1}|\xi_{k,r}|^q
   \right)^{1/q}.
\]
Moreover, for $\beta\in((H+1/q)^{-1},\min\{\alpha,q\})$, we have by H\"{o}lder's inequality,
\[
\begin{aligned}
   \mathbb E [B_{k,q}^{\beta}]
   &\le
   p^{-(k+1)\beta/q}
   \sum_{r=0}^{p^{k+1}-1}\mathbb E[|\xi_{k,r}|^\beta]
   =
   p^{(k+1)(1-\beta/q)}\mathbb E[|\xi|^\beta] .
\end{aligned}
\]
Therefore,
\[
   \sum_k \mathbb E\bigl[(p^{-kH}B_{k,q})^\beta\bigr]
   \le
   C\sum_k p^{-k(\beta H+\beta/q-1)}<\infty 
\]
for some $C>0$. 
This implies
\begin{equation}
     \sum_k p^{-kH}B_{k,q}<\infty
   \qquad\text{a.s.}\label{eq:c2}
\end{equation}
Indeed, for \(0<\beta\le1\), use
\((\sum a_k)^\beta\le\sum a_k^\beta\); for \(\beta\ge1\), use
Minkowski's inequality in \(L^\beta\).
Combining \eqref{eq:c1} and \eqref{eq:c2}, we arrive at
\begin{equation}\label{eq:B-tail}
  \sum_{k=K}^{\infty}p^{-kH}A_{k,q}(\tau)\leq 2 \sum_{k=K}^{\infty}p^{-kH}B_{k,q}\to 0\qquad\text{a.s.}
\end{equation}
as \(K\to\infty\). Combining 
\eqref{eq:weyl-upper-bound} and  \eqref{eq:B-tail} yields that a.s., for every \(\varepsilon>0\), there exists some \(K\) with the property that every \(\tau\in p^K\N\) is a Weyl
\(q\)-\(\varepsilon\)-translation number.
Since \(p^K\N\) has bounded gaps, \(X\) is Weyl \(q\)-almost periodic.
Besicovitch \(q\)-almost periodicity follows because the Besicovitch
seminorm is no larger than the Weyl seminorm.

On the other hand, using $H<1/\alpha$, Corollary 4 of \cite{Zhang2026DiscountedBRW} implies that the process $X$ is not sample bounded, and hence not $p$-adically continuous. This completes the proof.
\end{proof}

\section{Random Field Extensions}\label{sec:multi}

This section extends some previous results to the random field setting. We first record the corresponding definitions of $p$-adic sssi random fields, $p$-adic continuity, and almost periodicity.

Let $d\geq 1$. For a process indexed by \(\Nzero^d\), let \(0\) denote
the zero vector.  A real-valued \emph{\(p\)-adic \(H\)-sssi random field} is
a process \(X=(X_n)_{n\in\Nzero^d}\) such that
\begin{equation*}
   (X_{a n})_{n\in\Nzero^d}\dd |a|_p^H(X_n)_{n\in\Nzero^d},
   \qquad
   (X_{n+\ell}-X_\ell)_{n\in\Nzero^d}\dd (X_n-X_0)_{n\in\Nzero^d}
\end{equation*}
for all \(a\in\N\) and \(\ell\in\Nzero^d\).  
We say a deterministic function $f:\Nzero^d\to\R$ is 
\emph{(uniformly) \(p\)-adically continuous} if
\begin{equation}\label{eq:padic-cont-d}
   \forall\varepsilon>0,\ \exists K,\quad
   \sup_{n,m\in\Nzero^d}|f(n+p^K m)-f(n)|<\varepsilon .
\end{equation}
A set \(T\subset\Nzero^d\) is \emph{relatively dense} if there is
\(L\in\N\) such that every box \(a+\{0,\dots,L\}^d\), \(a\in\Nzero^d\) contains an element of \(T\).  A function \(f:\Nzero^d\to\R\) is
\emph{Bohr almost periodic}, or simply \emph{almost periodic}, if for every
\(\varepsilon>0\) the set
\[
   \left\{h\in\Nzero^d:
   \sup_{n\in\Nzero^d}|f(n+h)-f(n)|<\varepsilon\right\}
\]
is relatively dense in \(\Nzero^d\); see~\cite{Bochner1933}.

\sloppy We proceed along a route similar to the proof of Theorem \ref{thm:main1}, starting from finite-dimensional analogues of
Lemmas~\ref{lem:finite1} and \ref{lem:arith1}. 

\begin{lemma}\label{lem:finite_d}
Let \(d\geq 1\), and let \(f:\Nzero^d\to\R\) be almost periodic.  Then
\(f\) is bounded.  Moreover, for every \(\eta>0\) there is
\(L=L(f,\eta)\) such that, for every \(h\in\Nzero^d\), some
\(r\in\{0,\dots,L\}^d\) satisfies
\begin{equation}\label{eq:finite-reduction-d}
   \sup_{m\in\Nzero^d}
   |(f(m+h)-f(h))-(f(m+r)-f(r))|<\eta .
\end{equation}
\end{lemma}

\begin{proof}
By the Bochner compactness criterion~\cite[Chapter~1]{Corduneanu1989}, the translation orbit
\[
   \mathcal T(f)=\{m\mapsto f(m+h):h\in\Nzero^d\}
\]
is relatively compact in \(\ell^\infty(\Nzero^d)\).  In particular, \(f\)
is bounded.  Given \(\eta>0\), choose a finite \(\eta/2\)-net for
\(\mathcal T(f)\) consisting of translates \(m\mapsto f(m+r_j)\), where $r_j\in\N_0^d$, \(1\leq j\leq J\).  Let \(L=\max_{j,i}(r_j)_i\).  For each \(h\), choose
\(j\) such that
\[
   \sup_{m\in\Nzero^d}|f(m+h)-f(m+r_j)|<\frac{\eta}{2} .
\]
Taking \(m=0\) also gives \(|f(h)-f(r_j)|<\eta/2\).  Therefore,
\eqref{eq:finite-reduction-d} holds with \(r=r_j\) by the triangle inequality.
\end{proof}

\begin{lemma}\label{lem:sublattice}
If \(f:\Nzero^d\to\R\) is almost periodic, then, for every
\(r\in\Nzero^d\) and \(K\geq0\),
\[
   (T_{r,K}f)(m)=f(r+p^K m)-f(r),\qquad m\in\Nzero^d,
\]
is almost periodic.
\end{lemma}

\begin{proof}
Use the finite-dimensional Banach space-valued Bohr approximation theorem:
\(f\) is a uniform limit on \(\Nzero^d\) of trigonometric polynomials
\[
   P(n)=\sum_{j=1}^J b_j e^{2\pi i\langle \lambda_j,n\rangle},\qquad b_j\in \mathbb{C}.
\]
For such a polynomial,
\[
   P(r+p^K m)-P(r)
   =
   \sum_{j=1}^J b_j e^{2\pi i\langle \lambda_j,r\rangle}
   \bigl(e^{2\pi i\langle p^K\lambda_j,m\rangle}-1\bigr),
\]
again a trigonometric polynomial in \(m\).  Passing to uniform limits gives
the claim.
\end{proof}

\begin{theorem}\label{thm:field}
Let \(d\geq 1\) and let \(X=(X_n)_{n\in\Nzero^d}\) be a real-valued
\(p\)-adic \(H\)-sssi random field.  Then
\[
\begin{aligned}
&\Pp(\{X\hbox{ is almost periodic on }\Nzero^d\}\Delta\{X\hbox{ is }p\hbox{-adically continuous on }\Nzero^d\})=0.
\end{aligned}
\]
\end{theorem}

\begin{proof}
The deterministic implication from \(p\)-adic continuity to almost
periodicity is the same as in Theorem~\ref{thm:main1}: if
\eqref{eq:padic-cont-d} holds, then every
\(h\in(p^K\Nzero)^d\) is an \(\varepsilon\)-translation vector, and
\((p^K\Nzero)^d\) is relatively dense in \(\Nzero^d\).

For the converse, we repeat the proof of Theorem~\ref{thm:main1}, with
Lemma~\ref{lem:finite_d} replacing Lemma~\ref{lem:finite1},
Lemma~\ref{lem:sublattice} replacing Lemma~\ref{lem:arith1}, and the maps
\[
   \Phi_{r,K}X=(X_{r+p^K m}-X_r)_{m\in\Nzero^d}
\]
replacing the one-dimensional version.  The finite union over
\(r\in\{0,\dots,L_0\}\) is replaced by the finite union over
\(r\in\{0,\dots,L_0\}^d\).  This gives the other direction of the theorem.
\end{proof}

\end{document}